\documentclass[10pt]{amsart}

\usepackage{amssymb,amsmath,amscd,amsthm,amsfonts,wasysym,mathrsfs,amsxtra}
\input xy
\xyoption{all}

\newcommand{\OO}{\mathscr{O}}

\newcommand{\QQ}{\mathbb{Q}}

\newcommand{\Cfield}{\mathbb{C}}

\newcommand{\PP}{\mathbb{P}}

\newcommand{\image}{\textnormal{im}\,}
\newcommand{\kernel}{\textnormal{ker}\,}
\newcommand{\cokernel}{\textnormal{coker}\,}

\newcommand{\HHom}{\mathscr{H}om} %for sheaf
\newcommand{\Hom}{\textnormal{Hom}}

\newcommand{\rank}{\textnormal{rk}}
\newcommand{\Pic}{\textnormal{Pic}}

\newcommand{\Ext}{\textnormal{Ext}}
\newcommand{\EExt}{\mathscr{E}xt}

\newcommand{\al}{\alpha}

\newcommand{\ch}{\textnormal{ch}}

\newcommand{\Coh}{\textnormal{Coh}}

\newcommand{\arinj}{\ar@{^{(}->}}
\newcommand{\arsurj}{\ar@{->>}}
\newcommand{\onto}{\twoheadrightarrow}
\newcommand{\into}{\hookrightarrow}

\newcommand{\nuob}{\nu_{\omega, B}}
\newcommand{\numob}{\nu_{m\omega, B}}
\newcommand{\muob}{\mu_{\omega,B}}

\newcommand{\nuz}{\nu_{\omega, 0}}

\newcommand{\om}{\omega}
\newcommand{\Bob}{\mathcal B_{\omega, B}}
\newcommand{\Tob}{\mathcal T_{\omega, B}}
\newcommand{\Fob}{\mathcal F_{\omega, B}}
\newcommand{\Aob}{\mathcal A_{\omega, B}}

\newcommand{\tch}{\widetilde{ch}}

\newcommand{\NS}{\textnormal{NS}}
\newcommand{\rk}{\textnormal{rk}}
\newcommand{\omB}{{_{\om, B}}}  %subscript {\om, B}

\makeatletter
\newtheorem*{rep@theorem}{\rep@title}
\newcommand{\newreptheorem}[2]{%
\newenvironment{rep#1}[1]{%
 \def\rep@title{#2 \ref{##1}}%
 \begin{rep@theorem}}%
 {\end{rep@theorem}}}
\makeatother

\newtheorem{theorem}{Theorem}[section]
\newreptheorem{theorem}{Theorem}
\newreptheorem{proposition}{Proposition}

\newtheorem{lemma}[theorem]{Lemma}
\newtheorem{pro}[theorem]{Proposition}
\newtheorem{proposition}[theorem]{Proposition}

\newtheorem{corollary}[theorem]{Corollary}

\theoremstyle{definition}

\newtheorem{example}[theorem]{Example}
\theoremstyle{remark}
\newtheorem{remark}[theorem]{Remark}
\newtheorem{conjecture}[theorem]{Conjecture}

\numberwithin{equation}{section}

\begin{document}

\title{Some examples of tilt-stable objects on threefolds}
%    Information for first author
\author{Jason Lo}
%    Address of record for the research reported here
\address{Department of Mathematics, University of Missouri-Columbia, Columbia, Missouri 65211}
\email{locc@missouri.edu}

%    Information for second author
\author{Yogesh More}
\address{Department of Mathematics, SUNY College at Old Westbury, Old Westbury, NY 11568}
\email{yogeshmore80@gmail.com}
%    General info
%\subjclass[2000]{Primary 54C40, 14E20; Secondary 46E25, 20C20}

\begin{abstract}
We investigate properties and describe examples of tilt-stable objects on a smooth complex projective threefold.  We give a structure theorem on slope semistable sheaves of vanishing discriminant, and describe certain Chern classes for which every slope semistable sheaf yields a Bridgeland semistable object of maximal phase.  Then, we study tilt stability as the polarisation $\omega$ gets large, and give sufficient conditions for tilt-stability of sheaves of the following two forms: 1) twists of ideal sheaves or 2) torsion-free sheaves whose first Chern class is twice a minimum possible value.
\end{abstract}

\maketitle

\tableofcontents

\section{Introduction}

Let $X$ be a smooth projective threefold over $\Cfield$ throughout, unless otherwise stated. It has been a long standing open problem to construct a Bridgeland stability condition on an arbitrary  Calabi-Yau threefold. In \cite{BMT}, this problem is reduced to showing a Bogomolov-Gieseker type inequality involving $ch_3$ for a class of objects they call tilt-stable objects. And in \cite{BMT} and \cite{Mac}, this conjecture is proven for $X=\PP^3$.  The purpose of this paper is to give some examples of tilt-stable objects. There are at least two possible uses of specific examples of tilt stable objects: first to investigate the $ch_3$ bound conjectured in \cite{BMT}, and second, for understanding moduli spaces of Bridgeland stable objects.

We now give some details of the constructions introduced in \cite{BMT}. Let $\om, B$ be two numerical equivalence classes of $\mathbb{Q}$-divisors on $X$, with $\om$ an ample class. Motivated by formulas for central charges arising in string theory, one defines a function $Z_\omB:D^b(X) \to \Cfield$ on the bounded derived category $D^b(X)$ of coherent sheaves on $X$ by

\begin{align}
Z_{\om, B}(E)&=-\int_X e^{-B -i\om} \ch(E) \\
&=(-\tch_3(E) + \frac{\om^2}{2}\tch_1(E)) + i(\om\tch_2(E) - \frac{\om^3}{6}\tch_0(E))
\end{align}
where $\tch$ denotes the twisted Chern character $\tch(E)=e^{-B}ch(E)$. In \cite{BMT}, the function $Z_\omB$, along with an abelian category $\Aob$ that is the heart of a t-structure on $D^b(X)$, is conjectured to form a Bridgeland stability condition on $D^b(X)$, for any smooth projective threefold $X$ over $\Cfield$.

The heart $\Aob$ is constructed by a sequence of two tilts, starting with the abelian category $\Coh (X)$.  After a tilt of $\Coh (X)$, the paper \cite{BMT} defines a slope function $\nuob$ on the resulting heart $\Bob$and says an object in $\Bob$ is  ``tilt-(semi)stable'' if it is $\nuob$-(semi)stable.

We now describe the results in this article. In Section \ref{sec.reflexive}, we show that if  $E \in \Bob$ is a $\nuob$-semistable object with $\nuob(E) < \inf$, then $H^{-1}(E)$ must be a reflexive sheaf (Proposition \ref{coro1}).  This allows us to use results on reflexive sheaves in studying tilt-semistable objects.  For $E\in D^b(X)$, we can consider the discriminant in the sense of Dr\'{e}zet: $\overline{\Delta}_{\omega}(E) := (\om^2\tch_1(E))^2 - 2(\om^3\tch_0(E))(\om\tch_2(E)).$ In \cite[Proposition 7.4.1]{BMT}, it is shown that if $E$ is a slope-stable vector bundle on $X$ with $\overline{\Delta}_\om(E)=0$, then $E$ is tilt-stable. We show a partial converse to this:

\begin{reptheorem}{theorem2}
Suppose $E \in \Bob$ satisfies all of the following three conditions:
\begin{enumerate}
\item[(1)] $H^{-1}(E)$ is nonzero, torsion-free, $\muob$-stable (resp.\ $\muob$-semistable), with $\omega^2 \tch_1 (H^{-1}(E)) <0$;
\item[(2)] \label{thm2.two}$H^0(E) \in \Coh^{\leq 1}(X)$;
\item[(3)] \label{thm2.three}$\overline{\Delta}_\omega (E)=0$.
\end{enumerate}
Then $E$ is tilt-stable (resp.\ tilt-semistable) if and only if $E=H^{-1}(E)[1]$ where $H^{-1}(E)$ is a locally free sheaf.
\end{reptheorem}

Using the above theorem, we also obtain a better understanding of slope semistable sheaves of zero discriminant:

\begin{reptheorem}{theorem3}
Suppose $B=0$.  Let $F$ be a $\mu_\omega$-semistable sheaf with $\overline{\Delta}_\omega (F)=0$.  Then $\EExt^1 (F,\OO_X)$ is zero, and $F^\ast$ is locally free.  Therefore, $F$ is locally free if and only if the 0-dimensional sheaf $\EExt^2 (F,\OO_X)$ is zero.
\end{reptheorem}
As a corollary, we show how every slope semistable sheaf of zero discriminant and zero tilt-slope yields a $Z_{\omega, 0}$-semistable object of maximal phase in $\mathcal A_{\omega, 0}$:

\begin{reptheorem}{theorem4}
Suppose $F$ is a $\mu_\omega$-semistable sheaf with $\overline{\Delta}_\omega (F)=0$, $\nu_\omega (F)=0$  and $\omega^2 ch_1 (F)>0$.  Then $F^\vee [2]$ is an object of phase 1 with respect to $Z_{\omega,0}$ in $\mathcal{A}_{\omega,0}$.
\end{reptheorem}
Since taking derived dual and shift both preserve families of complexes, Theorem \ref{theorem4} implies that the moduli of $Z_{\omega, 0}$-semistable objects in $\mathcal A_{\omega, 0}$ with the prescribed Chern classes (if it exists) contains the  moduli of $\mu_\omega$-semistable sheaves  as an open subspace.  In the case of rank-one objects, for example, the open subspace contains the Hilbert scheme of points (see Remark \ref{remark1}).

In Section \ref{sec.infinity}, we analyse tilt-stability at the large volume limit.  In Remarks \ref{remark2} and \ref{remark3}, we mention the connections between tilt-semistable objects and polynomial semistable objects.

In section \ref{sec.2c}, we give some sufficient conditions for a torsion-free sheaf  $E\in \Tob$ with $\om^2\tch_1(E) =2c$ to be tilt-stable. Here, the number $c$ is defined in \cite[Lemma 7.2.2]{BMT} as
   \[
   c:=\min \{ \om^2\tch_1(F)>0 \mid F \in \Bob\}.
   \]
Tilt-semistable objects with $\omega^2 \tch_1 = c$ were already characterised in \cite{BMT}.  Our results include:

\begin{repproposition}{lemma.2c}
Suppose $E \in \Tob$ is a torsion-free sheaf with $\nuob(E)=0$ and $\om^2\tch_1(E)=2c$, where $c$ is defined above.
\begin{enumerate}
\item \label{lemma.2c.part1} If  $\mu_{\om,B,\max}(E) < \frac{\om^3}{\sqrt{3}}$, then $E$ is $\nuob$-stable.

\item \label{lemma.2c.part2} If $\om^3>3\om(\tch_1(M))^2$ for every torsion free slope semistable sheaf $M$ with $\om^2\tch_1(M)=c$, then $E$ is $\nuob$-stable.
\end{enumerate}
\end{repproposition}
We then apply this proposition to studying the tilt-stability of rank one torsion free sheaves that are twists of ideal sheaves of curves.

Finally, in Section \ref{sec.miro-roig}, we use known inequalities between Chern characters of reflexive sheaves on $\PP^3$ to describe many rank $3$ slope-stable reflexive sheaves $E \in \Bob$ that are tilt-unstable. (An object $E \in \Bob$ is defined to be tilt-unstable if it is not tilt-semistable.) We give examples illustrating  an observation  in \cite[p.4]{BMT}, that there are semistable sheaves on $\PP^3$ with $\nu(E)=0$ that do not satisfy $\tch_3(E) \leq \frac{\om^2}{18}\tch_1(E)$ (the inequality in Conjecture \ref{conj.bmt.1.3.1}). Since Conjecture \ref{conj.bmt.1.3.1} has been proven for $X=\PP^3$ (\cite{BMT}, \cite{Mac}), it follows that such $E$ must be tilt-unstable. This shows that the notion of tilt-stability is a necessary hypothesis in Conjecture \ref{conj.bmt.1.3.1}.

{\bf Acknowledgments}: The authors would like to thank Ziyu Zhang for helpful discussions, and Emanuele Macr\`{i} for kindly answering our questions.

{\bf Notation}: We write $\Coh^{\leq i}(X) \subset \Coh(X)$ for the subcategory of coherent sheaves supported in dimension $\leq i$, and $\Coh^{\geq i+1}(X) \subset \Coh(X)$ for the subcategory of coherent sheaves that have no subsheaves supported in dimension $\leq i$.

\section{Preliminaries}

Throughout this article, $X$ will always be a smooth projective threefold, unless otherwise specified.

In this section, we recall constructions introduced in \cite{BMT}. Let us  fix $\om, B \in \NS(X)_\QQ$ in the Neron-Severi group, with $\om$ an ample class. The category $\Aob$ will be formed by starting with $\Coh(X)$ and tilting twice.

First, the twisted slope $\muob$ on $\Coh(X)$ is defined as follows. If $E \in \Coh(X)$ is a torsion sheaf, set $\muob(E)=+\infty$. Otherwise set

\begin{equation}
\muob(E)=\frac{\om^2\tch_1(E)}{\tch_0(E)} = \frac{\om^2(\ch_1(E)-B\rk(E))}{\rk(E)}.
\end{equation}

Following \cite[Section 3.1]{BMT}, we say $E \in \Coh(X)$ is $\muob$-(semi)stable if, for any $F\in \Coh(X)$ with  $0 \neq F \subsetneq E$, we have $\muob(F) < (\leq) \muob(E/F)$. Let $\mu_\omega=\mu_{\om, 0}$. Since $\muob(E)=\mu_\om(E) - B\om^2$, it follows $E \in \Coh(X)$ is $\muob$-(semi)stable if and only if it is $\mu_\om$-(semi)stable.

Let $\Tob \subset \Coh(X)$ be the category generated, via extensions, by $\muob$-semistable sheaves $E$ of slope $\muob(E) >0$, and let $\Fob \subset \Coh(X)$ be the subcategory generated by $\muob$-semistable sheaves of slope $\muob\leq 0$. Then $(\Tob, \Fob)$ forms a torsion pair, and define the abelian category $\Bob$ as the tilt of $\Coh(X)$ with respect to $(\Tob, \Fob)$:

$$\Bob = \langle \Fob[1], \Tob \rangle.$$

For $E \in \Bob$, define its tilt-slope $\nuob(E)$ as follows. If $\om^2\tch_1(E)=0$, then set $\nuob(E)=+\infty$. Otherwise set
\begin{equation}
\nuob(E)=\frac{\Im Z\omB(E)}{\om^2\tch_1(E)}=\frac{\om\tch_2(E)-\frac{\om^3}{6}\tch_0(E)}{\om^2\tch_1(E)}.
\end{equation}

An object $E \in \Bob$ is defined to be $\nuob$-(semi)stable if, for any non-zero proper subobject $F \subset E$ in $\Bob$, we have $\nuob(F) < (\leq) \nuob(E/F)$. We will use tilt-(semi)-stability and $\nuob$-(semi)stability interchangably.

Let $\mathcal{T}_{\omega,B}'$ (resp. $\mathcal{F}_{\omega,B}'$) be the extension closed subcategory of $\Bob$ generated by $\nuob$-stable objects $E \in \Bob$ of tilt-slope $\nuob(E) >0$ (resp. $\nuob(E) \leq 0$). Then $(\mathcal{T}_{\omega,B}', \mathcal{F}_{\omega,B}')$ form a torsion pair in $\Bob$, and tilting $\Bob$ with respect to $(\mathcal{T}_{\omega,B}', \mathcal{F}'_{\omega,B})$ defines an abelian category $\Aob=\langle \mathcal{F}_{\omega,B}'[1], \mathcal{T}_{\omega,B}' \rangle$.

In \cite{BMT}, it is shown that $\sigma=(Z_\omB, \Aob)$ defines a Bridgeland stability condition as long as the image of the function $Z_\omB$ restricted to $\Aob\setminus\{0\}$ lies in the half-closed upper half plane $\mathcal{H}=\{z \in \mathbb{C} \mid \Im z>0, \textnormal{or } [\Im z =0 \textnormal{ and } \Re z < 0 ]\}$. For $E \in \Aob$, it follows automatically from the construction of $\Aob$ that $\Im Z_\omB(E) \geq 0$; the difficulty so far is verifying that $\Re Z_\omB(E) <0$ when $\Im Z_\omB(E)=0$. To be more precise, in \cite[Cor. 5.2.4]{BMT}, it is shown that $\sigma=(Z_\omB, \Aob)$ is a Bridgeland stability condition on $D^b(X)$ if and only if the following conjecture holds:

\begin{conjecture}\cite[Conjecture 3.2.6]{BMT}
Any tilt-stable object $E \in \Bob$ with $\nuob(E)=0$ satisfies
\begin{equation}
\tch_3(E) < \frac{\om^2}{2}\tch_1(E).
\end{equation}
\end{conjecture}

In fact, an even stronger inequality is conjectured in \cite{BMT}:

\begin{conjecture}\label{conj.bmt.1.3.1}\cite[Conjecture 1.3.1]{BMT}
Any tilt-stable object $E \in \Bob$ with $\nuob(E)=0$ satisfies
\begin{equation}\label{bmt-conj}
\tch_3(E) \leq \frac{\om^2}{18}\tch_1(E).
\end{equation}
\end{conjecture}

In \cite{BMT} and \cite{Mac}, this conjecture is proven for $\PP^3$, by using the fact that $\PP^3$ has a full strong exceptional collection.

\section{Reflexive sheaves and objects of zero discriminant}\label{sec.reflexive}

In \cite[Proposition 7.4.1]{BMT}, it is shown that any slope stable vector bundle with zero discriminant is a tilt-stable object.  The first goal of this section is to prove a partial converse to this result (Theorem \ref{theorem2}).  As a corollary, we produce a structure theorem on slope semistable sheaves of zero discriminant (Theorem \ref{theorem3}).  As another corollary, we show how, given any slope semistable sheaf of zero discriminant and $\nuob=0$, we can produce a $Z_{\omega, 0}$-semistable object in $\mathcal A_{\omega, 0}$ of maximal phase (Theorem \ref{theorem4}).  This implies that the Hilbert scheme of points on $X$ in contained in a moduli of Bridgeland semistable objects on $X$ if the moduli exists (Remark \ref{remark1}).

We begin with the following link between reflexive sheaves and $\nuob$-semistable objects in $\Bob$:

\begin{proposition}\label{coro1}
If $E \in \Bob$ is a $\nuob$-semistable object with $\nuob (E) < +\infty$, then $H^{-1}(E)$ is a reflexive sheaf.
\end{proposition}

The proof of this proposition relies on:

\begin{lemma}\label{lemma.doubledual}
Let $F \in \Coh(X)$ be a torsion-free sheaf, and let $F_n \in \Coh(X)$ be the Harder-Narasimhan $\muob$-semistable factor of $F$ with greatest $\muob$-slope. If $Q$ is the Harder-Narasimhan $\muob$-semistable factor of $F^{\ast\ast}$ with greatest $\muob$-slope, then $\muob(Q)=\muob(F_n)$. Hence if $F \in \Fob$, then $F^{\ast\ast} \in \Fob$.
\end{lemma}

\begin{proof}
 Observe that $F_n^{\ast\ast}$ is a $\muob$-semistable sheaf, and we have a canonical inclusion $F_n^{\ast\ast} \overset{\iota}{\into} F^{\ast\ast}$. Hence $\muob(Q) \geq \muob(F_n^{\ast\ast})$, by the proof of existence of HN filtrations. Let $Q \overset{\alpha}{\into} F^{\ast\ast}$ be the inclusion,  $F^{\ast\ast}\overset{\beta}{\to} T$ be the cokernel of $\iota$, and $K=\ker \beta\alpha$. We have a commutative diagram with exact rows and columns:

\begin{equation}
\xymatrix{&0 \ar[d] &0\ar[d] \\0\ar[r] &K \ar[r] \ar[d]_k &F \ar[d]^{\iota} \\
0 \ar[r] &Q \ar[r]^{\alpha} \ar[d]^{\beta\alpha} &F^{\ast\ast} \ar[d]^{\beta} \\
&T \ar[r]^{=} & T  \ar[d]\\
&&0}
\end{equation}

We have $\muob(K)=\muob(Q)$ (since $Q/K \subset T$ has codimension at least two), and $\muob(F_n) \geq\muob(K)$ (because $\muob(F_n)$ is greater than or equal to the slope of any subsheaf of $F$). Hence $\muob(F_n) \geq \muob(Q)$. Combined with $\muob(Q) \geq \muob(F_n^{\ast\ast})=\muob(F_n)$ we have $\muob(F_n) = \muob(Q)$.

The final statement follows from the definition that a sheaf $F \in \Coh(X)$ is in $\Fob$ if and only if $\mu_{\omega, B;\, \text{max}}(F):=\muob(F_n) \leq 0$.
\end{proof}

\begin{proof}[Proof of Proposition \ref{coro1}]
By Lemma \ref{lemma.doubledual}, we have $H^{-1}(E)^{\ast \ast} \in \Fob$.  Hence the canonical short exact sequence $0 \to H^{-1}(E) \to H^{-1}(E)^{\ast \ast} \to T \to 0$ gives us an injection $T \hookrightarrow H^{-1}(E)[1]$ in $\Bob$.  Together with the injection $H^{-1}(E)[1] \hookrightarrow E$ in $\Bob$, we get $T \hookrightarrow E$ in $\Bob$ where $T \in \Coh^{\leq 1}(X)$. If $T \neq 0$, then $\nuob(T)=\infty > \nuob(E)$, contradicting the $\nuob$-semistability of $E$. Hence $T=0$, i.e.\ $H^{-1}(E)$ must be a reflexive sheaf.
\end{proof}

\begin{corollary}
Let $F$ be a torsion free sheaf with $F[1] \in \Bob$. If $F[1]$ is $\nuob$-semistable, then $F$ is reflexive.
\end{corollary}

\begin{lemma}\label{lemma-coh0-closed}
The subcategory $\Coh^{\leq 0}(X)$ of $\Bob$ is closed under quotients, subobjects and extensions.
\end{lemma}
\begin{proof}
Given any short exact sequence $0 \to K \to Q \to B \to 0$ in $\Bob$ where $Q \in \Coh^{\leq 0}(X)$, consider the long exact sequence
\[
0 \to H^{-1}(B) \to H^0(K) \to H^0(Q) \to H^0(B) \to 0.
\]
If $H^{-1}(B)$ is nonzero, then it has positive rank, as does $H^0(K)$.  However, then $0 < \omega^2 \tch_1 (H^0(K)) = \omega^2 \tch_1(H^{-1}(B)) \leq 0$, which is a contradiction.  Thus $H^{-1}(B)=0$, and the lemma follows.
\end{proof}

The next proposition roughly says that modifying an object in codimension 3 does not alter its $\nuob$-(semi)stability:

\begin{proposition}\label{lemma11}
Suppose we have a short exact sequence in $\Bob$
\begin{equation}\label{eq7}
0 \to E' \to E \to Q \to 0
\end{equation}
where $Q \in \Coh^{\leq 0}(X)$.
\begin{enumerate}
\item \label{case-Ess} If $E$ is $\nuob$-semistable (resp.\ $\nuob$-stable), then $E'$ is $\nuob$-semistable (resp.\ $\nuob$-stable).
\item \label{case-E'ss} Assuming $\Hom (\Coh^{\leq 0}(X),E)=0$, if $E'$ is $\nuob$-semistable, then $E$ is $\nuob$-semistable.
\item \label{case-E's} Assuming $\Hom (\Coh^{\leq 0}(X),E)=0$ and $\omega^2 \tch_1 (E) \neq 0$, if $E'$ is $\nuob$-stable then $E$ is $\nuob$-stable.
\item \label{case-conj} If $E$ satisfies Conjecture \ref{bmt-conj} then $E'$ also satisfies the same conjecture.
\end{enumerate}
\end{proposition}
% [Does the converse to part \ref{case-conj} (assuming also the hypotheses in part \ref{case-E's}) hold?]

\begin{proof}
Consider a commutative diagram of the form

\xymatrix{
 & 0 \ar[d] & 0 \ar[d] & & \\
 & A' \ar[r]^\al \ar[d]^{\gamma'} & A \ar[d]^\gamma & & \\
 0 \ar[r] & E' \ar[r]^e \ar[d]^{\delta'} & E \ar[r] \ar[d]^\delta & Q \ar[r] & 0 \\
 & B' \ar[r]^\beta \ar[d] & B \ar[d] & & \\
 & 0 & 0 & &}
\noindent where the row is the exact sequence \eqref{eq7}, and both columns are short exact sequences in $\Bob$.

\textit{Proof of part \ref{case-Ess}.} Suppose $A'$ is a nonzero proper subobject of $E'$.  We can put  $A=A'$, $\al = \text{id}_{A'}$, $\gamma = e\gamma'$, and let $\beta$ be the induced map of cokernels from the upper commutative square.  Then by the snake lemma in the abelian category $\Bob$, $\cokernel (\beta)$ is a quotient of $Q$ in $\Bob$, and hence is a 0-dimensional sheaf by Lemma \ref{lemma-coh0-closed}, while $\kernel (\beta)=0$.  Thus $\nuob (B')=\nuob (B)$.  We also have $\nuob (A')=\nuob (A)$ (since $A'=A$).  Note that $A$ is a nonzero proper subobject of $E$. If $E$ is $\nuob$-semistable, then $\nuob (A) \leq \nuob (B)$, implying $\nuob (A') \leq \nuob (B')$, and hence $E'$ is $\nuob$-semistable.  Similarly, if $E$ is $\nuob$-stable, then $E'$ is also $\nuob$-stable.

\textit{Proof of part \ref{case-E'ss}.} Suppose that $A$ is a nonzero proper subobject of $E$.  We can put $B' = \image (\delta e)$, $\delta' = \delta e$, $A' = \kernel (\delta')$, put $\beta$ as the canonical inclusion $\image (\delta') \hookrightarrow B$, and put $\al$ as the induced map of kernels from the lower commutative square.  If $A'=0$, then $\delta'$ is an isomorphism.  However, this implies that $\delta$ restricts to an injection from $E'$, i.e.\ $E' \cap A =0$.  Hence the quotient $E \twoheadrightarrow Q$ induces an injection $A \hookrightarrow Q$, so $A \in \Coh^{\leq 0}(X)$ by Lemma \ref{lemma-coh0-closed}, which contradicts our assumption $\Hom (\Coh^{\leq 0}(X),E)=0$.  Therefore, $A'$ is nonzero.

On the other hand, if $A'=E'$, then $\delta'$ is the zero map, meaning $E' \subset A$, and so there is a surjection $Q \twoheadrightarrow B$ in $\Bob$.  By Lemma \ref{lemma-coh0-closed},  $B \in \Coh^{\leq 0}(X)$, and hence $\nuob(B)=\infty$. So $\nuob (A) \leq \nuob(B)$ when $A'=E'$.

Now, suppose $A'$ is a nonzero proper subobject of $E'$.  Since $\al, e$ and $\beta$ are all injective maps, the snake lemma gives an induced short exact sequence in $\Bob$ of their cokernels:
\begin{equation}\label{eqn.cokernels}
0 \to \cokernel (\al) \to Q \to \cokernel (\beta) \to 0.
\end{equation}
Hence $\cokernel (\al), \cokernel (\beta)$ are both 0-dimensional sheaves by Lemma \ref{lemma-coh0-closed}, giving us $\nuob (A')=\nuob (A)$ and $\nuob (B')=\nuob (B)$.  If $E'$ is $\nuob$-semistable, then $\nuob (A') \leq \nuob (B')$, implying $\nuob (A) \leq \nuob (B)$, and hence $E$ is $\nuob$-semistable.

\textit{Proof of part \ref{case-E's}.} The proof is essentially same as for part \ref{case-E'ss}, with the following additional argument for the scenario $A'=E'$. If $A'=E'$, the hypothesis $0 \neq \omega^2 \tch_1 (E) = \omega^2 \tch_1 (E')$ along with the injection $E' \into A$ in $\Bob$ implies $\omega^2 \tch_1 (A) = \omega^2 \tch_1 (E')>0$ and hence $\nuob(A) <\infty=\nuob(B)$.

\textit{Proof of part \ref{case-conj}.} Assume $E$ is $\nuob$-stable, $\nuob(E)=0$, and $\tch_3(E) \leq \frac{\om^2}{18}\tch_1(E)$. Since the formula for $\nuob$ does not have any dependence on $\tch_3$, we have $\nuob(E)=\nuob(E')$, so $\nuob(E')=0$. By part \ref{case-Ess}, $E'$ is $\nuob$-stable. Finally, $$\tch_3(E')=\tch_3(E)-\tch_3(Q) \leq \tch_3(E) \leq \frac{\om^2}{18}\tch_1(E)=\frac{\om^2}{18}\tch_1(E').$$
\end{proof}

\begin{example}
Let $E$ be a $\muob$-stable vector bundle on $X$ with $\overline{\Delta}_\om(E)=0$. Then $E$ is $\nuob$-stable  by \cite[Prop. 7.4.1]{BMT}. Assume $\om^2\tch_1(E) >0$, so $E \in \Tob$. Begining with any surjection $E \onto Q$ in $\Coh(X)$ with $Q \in \Coh^{\leq 0}(X)$ we can apply Proposition \ref{lemma11} to obtain other examples of tilt-stable objects. For example, suppose $X$ has Picard number one. Then any line bundle $L$ on $X$ satisfies $\overline{\Delta}_\om(L)=0$. Choose a line bundle $L$ with $\om^2\tch_1(L) >0$. Let $I_Z$ be the ideal sheaf of any zero dimensional subscheme $Z \subseteq X$. Then applying Proposition \ref{lemma11} to the exact sequence $0 \to I_Z \otimes L \to L \to \mathcal{O}_Z \to 0$ shows $I_Z \otimes L$ is tilt-stable.
\end{example}

For objects $E \in D^b(X)$, we have the following two versions of discriminants (see \cite[Section 7.3]{BMT} for some background information):
\begin{enumerate}
\item  $\Delta(E):=(\tch_1(E))^2 - 2(\tch_0(E))(\tch_2(E))$, the definition that is usually used for coherent sheaves;
\item  $\overline{\Delta}_{\omega}(E) := (\om^2\tch_1(E))^2 - 2(\om^3\tch_0(E))(\om\tch_2(E))$.
%\item  $\Dthree=2(3(\om^3\rank)^2\tch_3 - 3(\om^3 \rank)(\om^2\tch_1)(\om\tch_2) + (\om^2\tch_1)^3 )$
\end{enumerate}
A calculation shows $\Delta(E)=(ch_1(E))^2 - 2(ch_0(E))(ch_2(E))$ that is, we may omit the tildes over the $ch_i$, and in particular the $\Delta(E)$ is independent of $B$. If the Picard number of $X$ is one, then  $\overline{\Delta}_\om$ is independent of $B$ \cite[Section 2.1]{Mac}, but in general  $\overline{\Delta}_\om$ depends on $B$.

For later use, we will need the following lemma.
\begin{lemma}\label{lem.compare-discr}
For any coherent sheaf $F$ on $X$, we have $\overline{\Delta}_\om(F) \geq (\om\Delta(F))\om^3$.
\end{lemma}

\begin{proof}
The Hodge Index Theorem gives $(\om^2\tch_1(F))^2\geq(\om^3)(\om\tch_1(F)^2)$, and hence
\begin{align}
\overline{\Delta}_\om(F) &= (\om^2\tch_1(F))^2 - 2(\om^3\tch_0(F))(\om\tch_2(F)) \\
&\geq \om^3(\om\tch_1(F)^2) - 2(\om^3\tch_0(F))(\om\tch_2(F))=\om^3(\om\Delta(F)).
\end{align}
\end{proof}

The following result was shown in \cite[Cor 7.3.2]{BMT}, and it was a key ingredient for the main result in \cite{Mac}.

\begin{proposition}\cite[Cor 7.3.2]{BMT}
If $E \in \Bob$ is $\nuob$-semistable, then $\overline{\Delta}_\om(E) \geq 0$.
\end{proposition}

In this section, we will investigate the tilt-stability of objects with $\overline{\Delta}_{\omega}(E)=0$. The following result gives many examples of tilt-stable objects. (Furthermore, in \cite[Proposition 7.4.2]{BMT}, they verify these objects satisfy Conjecture \ref{conj.bmt.1.3.1}, and equality holds).

\begin{proposition}\label{prop.bmt.discr0}\cite[Proposition 7.4.1]{BMT}
Let $E$ be a $\muob$-stable vector bundle on $X$ with $\overline{\Delta}_\om(E)=0$. Then $E$ is $\nuob$-stable.
\end{proposition}

 Now we come to the following partial converse to Proposition \ref{prop.bmt.discr0}.

\begin{theorem}\label{theorem2}
Suppose $E \in \Bob$ satisfies all of the following three conditions:
\begin{enumerate}
\item[(1)] $H^{-1}(E)$ is nonzero, torsion-free, $\muob$-stable (resp.\ $\muob$-semistable), with $\omega^2 \tch_1 (H^{-1}(E)) <0$;
\item[(2)] \label{thm2.two}$H^0(E) \in \Coh^{\leq 1}(X)$;
\item[(3)] \label{thm2.three}$\overline{\Delta}_\omega (E)=0$.
\end{enumerate}
Then $E$ is tilt-stable (resp.\ tilt-semistable) if and only if $E=H^{-1}(E)[1]$ where $H^{-1}(E)$ is a locally free sheaf.
\end{theorem}

\begin{remark}\label{remark2}
Note that, any polynomial stable complex on $X$ that is PT-semistable or dual-PT-semistable (see \cite{Lo3}) of positive degree satisfies conditions (1) and (2) in Theorem \ref{theorem2}.  However, the theorem says that, under the assumption $\overline{\Delta}_\omega =0$, a (dual-)PT-semistable object cannot be a genuine complex if it is to be tilt-semistable.
\end{remark}

We break up the proof of Theorem \ref{theorem2} into a couple of intermediate results.

\begin{proposition}\label{pro2}
Let $F$ be a $\muob$-semistable reflexive sheaf on $X$ such that $\overline{\Delta}_\omega (F)=0$.  Then $F$ is a locally free sheaf.
\end{proposition}

\begin{proof}
The proof is largely based on that of \cite[Proposition 7.4.2]{BMT}. By \cite[Theorem 4.1.10]{Laz}, we can find a pair $(f,L)$ where $f$ is a morphism $Y \to X$ that is finite, surjective and flat, with $Y$  a smooth projective variety, and a line bundle $L$ on $Y$ such that $(f^\ast \omega)^2 \tch_1 (L \otimes f^\ast F)=0$.

Since $f$ is flat and both $X,Y$ are smooth, $L \otimes f^\ast F$ is reflexive by \cite[Proposition 1.8]{SRS}.  On the other hand, by choosing $L$ above so that $c_1(L)$ is a rational multiple of $f^\ast \omega$, we have $\overline{\Delta}_{f^\ast\omega} (L \otimes f^\ast F) = \overline{\Delta}_{f^\ast\omega} (f^\ast F)=0$ because the discriminant $\overline{\Delta}_{f^\ast \omega}$ is invariant under tensoring by a line bundle whose $c_1$ is proportional to $f^\ast \omega$, and $\overline{\Delta}_\omega (F)=0$.  Hence $(f^\ast \omega)\tch_2 (L \otimes f^\ast F)=0$.  Passing to another finite cover of the form above, we can assume that $B$ is the divisor class of a line bundle $M$ on $Y$, and $ch (M \otimes L \otimes f^\ast F) = \tch(L \otimes f^\ast F)$.  Now, $f^\ast F$ is $\mu_{f^\ast \omega, f^\ast B}$-semistable since $f$ is a finite morphism.  Hence $M \otimes L \otimes f^\ast F$ is $\mu_{f^\ast \omega, f^\ast B}$-semistable (and equivalently, $\mu_{f^\ast\om}$-semistable) with vanishing $(f^\ast \omega)^2 ch_1$ and $(f^\ast \omega)ch_2$.  Thus, by \cite[Proposition 5.1]{Lan09}, $M \otimes L \otimes f^\ast F$ is locally free, i.e.\ $f^\ast F$ is locally free.  Since $f$ is surjective and flat, it is faithfully flat, and so $F$ itself is locally free.
\end{proof}

\begin{lemma}\label{lemma.split}
If $E \in \Bob$ with $H^{-1}(E)$ a vector bundle, and $H^0(E)\in \Coh^{\leq 0}(X)$, then $E\cong H^{-1}(E)[1] \oplus H^0(E)$.  If  $E$ further satisfies $\Hom(\Coh^{\leq 0}(X), E)=0$ or $E$ is $\nuob$-stable, then $H^0(E)=0$, in which case $E\simeq H^{-1}(E)[1]$ is a shift of a vector bundle.
\end{lemma}

\begin{proof}
Let $F=H^{-1}(E)$ and $T=H^0(E)$. We have $\Ext^1(T, F[1])=\Ext^2(T, F)=$ \newline $\Ext^1(F, T \otimes \om_X)=H^1(X, F^* \otimes T \otimes \om_X)$, which is zero since $T \in \Coh^{\leq 0}(X)$. From the exact sequence $F[1] \to E \to T$ in $\Bob$ we conclude $E\simeq F[1] \oplus T$. If $E$ is $\nuob$-stable, then $T=0$ (otherwise $T$ would be a $\nuob$-destabilizing object of $E$).
\end{proof}

\begin{proof}[Proof of Theorem \ref{theorem2}]
If $E=H^{-1}(E)[1]$ where $H^{-1}(E)$ is a $\muob$-stable (resp.\ $\muob$-semistable) locally free sheaf satisfying (1) through (3), then the result is \cite[Proposition 7.4.1]{BMT}.  (Note that, \cite[Proposition 7.4.1]{BMT} still holds if we replace each occurence of `stable' by `semistable' in its statement.)

 Now, assume $E$ satisfies (1) through (3) and is tilt-semistable. Let $F=H^{-1}(E)$. Then by Proposition \ref{coro1}, $F$ is reflexive. The condition $H^0(E) \in \Coh^{\leq 1}(X)$ implies $\om^3\tch_0(H^0(E))=\om^2\tch_1(H^0(E))=0$, and hence the condition $\overline{\Delta}_\omega (E)=0$ can be rewritten as
\begin{equation}\label{eqn.split.discrim}
\overline{\Delta}_\om(F) + 2\om^3\tch_0(F)\om\tch_2(H^0(E)) =0.
\end{equation}

The Bogomolov-Gieseker inequality says $\om\Delta(F) \geq 0$, and hence by Lemma \ref{lem.compare-discr}, we have $\overline{\Delta}_\om(F) \geq 0$. Since both terms  $\overline{\Delta}_\om(F)$ and $2\om^3\tch_0(F)\om\tch_2(H^0(E))$ are nonnegative, Equation \ref{eqn.split.discrim} implies they must both by zero. So $\om\tch_2(H^0(E))=0$, and $H^0(E) \in \Coh^{\leq 0}(E)$.  Since $\overline{\Delta}_\omega (F)=0$, we have $F$ is locally free by Proposition \ref{pro2}. By Lemma \ref{lemma.split}, we can conclude $E \simeq F[1]$.
\end{proof}

Using Theorem \ref{theorem2}, we can also prove the following result on $\mu_\omega$-semistable sheaves of zero discriminant:

\begin{theorem}\label{theorem3}
Suppose $B=0$.  Let $F$ be a $\mu_\omega$-semistable torsion-free sheaf with $\overline{\Delta}_\omega (F)=0$.  Then $\EExt^1 (F,\OO_X)$ is zero, and $F^\ast$ is locally free.  Therefore, $F$ is locally free if and only if the 0-dimensional sheaf $\EExt^2 (F,\OO_X)$ is zero.
\end{theorem}

To prove Theorem \ref{theorem3}, we first note:

\begin{lemma}\label{lemma12}
Suppose $B=0$. If $F$ is a $\mu_\omega$-semistable torsion-free sheaf on $X$ with $\overline{\Delta}_\omega (F)=0$, then $F$ must be locally free outside a codimension-3 locus.
\end{lemma}

\begin{proof}
Suppose the singularity locus of $F$ has codimension 2.  Then $\tch_2(F^{\ast \ast}/F) = ch_2 (F^{\ast \ast}/F) >0$, implying $\overline{\Delta}_\omega (F^{\ast \ast}) < \overline{\Delta}_\omega (F) =0$, which is a contradiction by  Lemma \ref{lem.compare-discr} and the usual Bogomolov-Gieseker inequality for $\mu_\omega$-semistable sheaves.  Hence the singularity locus of $F$ has codimension at least 3.
\end{proof}

%
%\begin{lemma}\label{lemma3}
%Let $X$ be a smooth projective threefold, and $E$ a reflexive sheaf on $X$.  Then
%\[
%  \Hom_X (\Coh^{\leq 1}(X),E[1])=0.
%\]
%
%\end{lemma}
%\begin{proof}
%Take any $T \in \Coh^{\leq 1}(X)$.  Then $\Hom (T,E[1]) \cong \Ext^1 (T,E) \cong \Ext^2 (E,T \otimes \omega_X)$ by Serre duality.  We have a surjection
%\[
%H^2 (X, \HHom (E,T \otimes \omega_X)) \twoheadrightarrow \Ext^2 (E,T \otimes \omega_X)
%\]
 %by \cite[Theorem 2.5]{SRS}.  Since $T$ is supported in dimension at most 1, we have the desired vanishing.
%\end{proof}
%
%\begin{lemma}\label{lemma14}
%Suppose $E$ is a 2-term complex $[E^{-1} \overset{d}{\to} E^0]$ (with $E^i$ at degree $i$) where $E^{-1}$ is a reflexive sheaf.
 %\begin{itemize}
 %\item[(a)] If  $E^0 \in \Coh^{\geq 2}(X)$, then $\Hom (\Coh^{\leq 1}(X), E)=0$.
 %\item[(b)] If $E^0 \in \Coh^{\geq 1}(X)$, then $\Hom (\Coh^{\leq 0}(X),E)=0$.
 %\end{itemize}
%\end{lemma}
%
%\begin{proof}
%This follows immediately from applying $\Hom (T,-)$ to the exact triangle
%\[
%  E^{-1} \to E^0 \to E \to E^{-1} [1],
%\]
%and Lemma \ref{lemma3}, where we take $T$ to be any sheaf in $\Coh^{\leq 1}(X)$ for part (a) (resp.\ $\Coh^{\leq 0}(X)$ for part (b)).
%\end{proof}

\begin{lemma}\label{lemma16}
Suppose $B=0$.  Suppose $F$ is a $\mu_{\omega}$-semistable (resp.\ $\mu_\omega$-stable) torsion-free sheaf, such that $\omega^2 ch_1 (F)>0$ and $\overline{\Delta}_\omega (F)=0$.  Then $(\tau^{\leq 1}F^\vee) [1]$ is a $\nu_{\omega,0}$-semistable (resp.\ $\nu_{\omega,0}$-stable) object.
\end{lemma}

\begin{proof}
By Lemma \ref{lemma12}, the sheaf $F$ is locally free outside a 0-dimensional locus.  Hence $\EExt^i (F,\OO_X)$ is 0-dimensional for all $i>0$, implying $\overline{\Delta}_\omega (F^\ast)=0$.  Since $F^\ast$ is reflexive,  Proposition \ref{pro2} implies $F^\ast$ is locally free.  And so $F^\ast [1]$ is $\nu_{\omega,0}$-semistable by \cite[Proposition 7.4.1]{BMT}.  Applying $\Hom (\Coh^{\leq 0}(X),-)$ to the exact triangle in $D(X)$
\begin{equation}\label{eq30}
  \tau^{\geq 2}(F^\vee) \to (\tau^{\leq 1}(F^\vee)) [1] \to F^\vee [1] \to \tau^{\geq 2}(F^\vee) [1]
\end{equation}
and writing $E := (\tau^{\leq 1}(F^\vee))[1]$, we obtain $\Hom (\Coh^{\leq 0}(X),E)=0$.  Hence, by applying Proposition \ref{lemma11} to the short exact sequence
\[
0 \to  F^\ast [1] \to E \to \EExt^1 (F,\OO_X) \to 0
\]
 in $\Bob$, we get that $E$ itself is $\nu_{\omega,0}$-semistable.
\end{proof}

We can now finish the proof of Theorem \ref{theorem3}:

\begin{proof}[Proof of Theorem \ref{theorem3}]
By tensoring $F$ with $\OO_X (m\omega)$ for $m \gg 0$, we can assume $\omega^2 \tch_1(F)>0$.  From the proof of Lemma \ref{lemma16}, we know $F^\ast$ is locally free. By Lemma \ref{lemma16} and Theorem \ref{theorem2}, we have $H^0( \tau^{\leq 1}(F^\vee))[1]=0$, i.e.\ $\EExt^1 (F,\OO_X)$ is zero.  The last assertion of Theorem \ref{theorem3} follows from the fact that any torsion-free sheaf on a smooth threefold has homological dimension at most 2.
\end{proof}

Recall the following easy consequence of \cite[Propositions 7.4.1, 7.4.2]{BMT}: suppose $F$ is  a $\muob$-stable vector bundle  on $X$ with $\overline{\Delta}_\omega (F)=0$ and $\nuob (F)=0$.  Then the object $F[1]$ (resp.\ $F[2]$) lies in $\Aob$, has phase 1 with respect to $Z_{\omega,B}$ and hence is $Z_{\omega,B}$-semistable if $\omega^2 \tch_1 (F) >0$ (resp.\ $\omega^2 \tch_1(F) \leq 0$).  Now we have a slight extension of this result:

\begin{theorem}\label{theorem4}
Suppose $F$ is a $\mu_\omega$-semistable sheaf with $\overline{\Delta}_\omega (F)=0$, $\nu_\omega (F)=0$  and $\omega^2 ch_1 (F)>0$.  Then $F^\vee [2]$ is an object of phase 1 with respect to $Z_{\omega,0}$ in $\mathcal{A}_{\omega,0}$.
\end{theorem}

 In particular, if $(\mathcal{A}_{\omega,0},Z_{\omega,0})$ is a stability condition, then we can speak of $F^\vee [2]$ as a $Z_{\omega,0}$-semistable object.

\begin{proof}
By Lemma \ref{lemma16}, we know $(\tau^{\leq 1}F^\vee)[1]$ is $\nu_{\omega,0}$-semistable with $\nu_{\omega,0}=0$.  Hence $(\tau^{\leq 1}F^\vee)[1] \in \mathcal{F}'_{\omega,0}$, and so $(\tau^{\leq 1}F^\vee)[2] \in \mathcal{A}_{\omega,0}$.  Since $(\tau^{\geq 2}F^\vee)[2]$ also lies in $\mathcal{A}_{\omega,0}$ and has phase 1 with respect to $Z_{\omega,0}$, from the exact triangle \eqref{eq30} we see that $F^\vee [2]$ is also  of phase 1 in $\mathcal{A}_{\omega,0}$.
\end{proof}

\begin{remark}\label{remark1}
Given Theorem \ref{theorem4}, it is reasonable to hope that for any Chern character $ch$ satisfying the conditions in the theorem, the moduli space of $Z_{\omega,0}$-semistable objects in $\mathcal{A}_{\omega,0}$ (provided $(\mathcal{A}_{\omega,0},Z_{\omega,0})$ is a stability condition and the moduli space  exists) contains the moduli of slope semistable sheaves of Chern character $ch$ as a subspace.

More concretely, suppose $Z \subset X$ is a 0-dimensional subscheme of length $n$, and let $L$ be a  line bundle on $X$ such that $I_Z \otimes L$ satisfies the hypotheses of Theorem \ref{theorem4}.  For instance, we can choose $L$ so that $c_1(L)$ is proportional to $\omega$ (so that tensoring $I_Z$ by $L$ does not alter its $\overline{\Delta}_\omega$); on the other hand, it can be checked easily that  $\nu_\omega (I_Z \otimes L)=0$ is equivalent to $3\omega c_1(L)^2 = \omega^3$, provided $\omega^2 c_1(L) \neq 0$.   Then $(I_Z \otimes L)^\vee [2]$ would be an object of of $\mathcal{A}_{\omega,0}$ with phase $1$ with respect to $Z_{\omega,0}$, and hence would be $Z_{\om, 0}$-semistable in $\mathcal{A}_{\omega,0}$.  Therefore, if the moduli space of $Z_{\omega,0}$-semistable objects $E \in \mathcal{A}_{\omega,0}$ with fixed chern character $ch(E)=ch((I_Z \otimes L)^\vee [2])$ exists, then it contains the Hilbert scheme of $n$ points on $X$.
%Note that,  if $I_Z$ is the ideal sheaf of any 0-dimensional subscheme $Z \subset X$, then $I_Z^\vee [2]$ is also an object of phase 1 with respect to $Z_{\omega,0}$ in $\Aob$ by \cite[Proposition 7.4.1]{BMT}.  Therefore, any Hilbert schemes of points on $X$ would be contained in the moduli space of Bridgeland semistable objects of the same Chern character, if the latter exists.
The following lemma shows that, under the condition $H^{-1}(E)=0$, a $Z_{\omega, 0}$-semistable object $E \in \mathcal{A}_{\omega,0}$ with the same Chern classes as $(I_Z \otimes L)^\vee [2]$ is `almost' (i.e.\ up to a 0-dimensional sheaf sitting at degree 0) of the form $(I_Z \otimes L)^\vee [2]$.
\end{remark}

\begin{lemma}\label{lemma17}
Suppose $B=0$, and any line bundle on $X$ with the same Chern classes as $\OO_X$ is isomorphic to $\OO_X$ (e.g.\ when $X$ has Picard rank 1).  Suppose $E \in \mathcal A_{\omega, 0}$ is such that $ch(E) = ch( (I_Z \otimes L)^\vee [2])$ where $I_Z, L$ are as in Remark 3.18.  (In particular, this means $\omega^2 ch_1(E) \neq 0$, $\nu_{\omega, 0} (E)=0$, and $Z_{\omega, 0}(E)$ has phase 1.)
If $H^{-1}(E)=0$, then $H^0(E^\vee [2]) \cong I_Y \otimes L$ where $I_Y$ is the  ideal sheaf of some 0-dimensional subscheme $Y$ of $X$, and $H^0(E)$ is a 0-dimensional sheaf.
\end{lemma}

\begin{proof}
With respect to $Z_{\omega, 0}$-stability, $E$ has a filtration in $\mathcal A_{\omega, 0}$ with $Z_{\omega, 0}$-stable factors $E^i$.  Since $\Im Z_{\omega, 0}(E)=0$, the same holds for each $E^i$.  For each $i$, we have a canonical short exact sequence in $\mathcal A_{\omega, 0}$
\[
 0 \to E_1^i [1] \to E^i \to E_2^i \to 0
\]
where $E_1^i \in \mathcal F_{\omega, 0}'$ and $E_2^i \in\mathcal T_{\omega,0}'$.  Since $E^i$ is $Z_{\omega, 0}$-stable, for each $i$, either $E^i =E_1^i[1]$ or $E^i = E_2^i$.

We now make an observation on objects in $\mathcal T_{\omega, 0}'$:  Suppose $G$ is any object in $\mathcal T_{\omega, 0}'$ with $\Im Z_{\omega, 0}(G)=0$.  Then $G$ is necessarily  $Z_{\omega, 0}$-semistable as an object in $\mathcal A_{\omega, 0}$.  With respect to $\nu_{\omega, 0}$-stability, $G$ has a filtration in $\mathcal B_{\omega, 0}$ with $\nu_{\omega, 0}$-stable factors $G^i$.  By the definition of $\mathcal T_{\omega, 0}'$, we know $\nu_{\omega, 0} (G^i)>0$ for each $i$.  On the other hand, each $G^i$ lies in $\mathcal T_{\omega, 0}' \subset \mathcal A_{\omega, 0}$, and so $G$ is an extension of the $G^i$ in $\mathcal A_{\omega, 0}$ as well.  Hence $\Im Z_{\omega, 0}(G^i)=0$ for all $i$.  Now,  if $\omega^2 ch_1 (G^i) \neq 0$ for some $i$, then $\omega^2 ch_1 (G^i) > 0$, and so $\Im Z_{\omega, 0}(G^i)>0$, which is a contradiction.  Hence $\omega^2 ch_1 (G^i)=0$ for all $i$.  By \cite[Remark 3.2.2]{BMT}, each $G^i$ lies in the extension-closed category
\[
 \mathcal C := \langle \Coh^{\leq 1}(X), F[1] : F \text{ $\mu_{\omega,0}$-stable with $\mu_{\omega,0} (F)=0$}\rangle \subset \mathcal B_{\omega,0}.
\]
Note that, every object in $\mathcal C$ has $\nu_{\omega, 0}=+\infty$, and is thus $\nu_{\omega, 0}$-semistable.  Hence $\mathcal C \subset \mathcal T_{\omega, 0}'$, and each $G^i$, being $\nu_{\omega, 0}$-stable, either lies in $\Coh^{\leq 1}(X)$ or is of the form $F[1]$ for some $\mu_{\omega,0}$-stable sheaf of $\mu_{\omega, 0}=0$.  Furthermore, if $G^i$ lies in $\Coh^{\leq 1}(X)$, then it must lie in $\Coh^{\leq 0}(X)$ since $\Im Z_{\omega, 0}(G^i)=0$.

Now, from the canonical short exact sequence
\begin{equation}\label{eq28}
0 \to E_1[1] \to E \to E_2 \to 0
 \end{equation}
 in $\mathcal A_{\omega, 0}$, we see that $H^{-1}(E)=0$ implies $H^0(E_1)=0$ and $H^{-1}(E_2)=0$.  That is, both $E_1, E_2$ are sheaves (up to shift).  In particular, by our observation above, $E_2$ must be an extension of objects in $\Coh^{\leq 0}(X)$, and so $H^0(E)\cong E_2 \in \Coh^{\leq 0}(X)$.

On the other hand, $H^{-1}(E_1)$ is a rank-one torsion-free sheaf by our assumption on $ch(E)$. Dualising \eqref{eq28} and shifting, we get an exact triangle
\begin{equation}\label{eq29}
  E_2^\vee [2] \to E^\vee [2] \to E_1^\vee [1].
\end{equation}
Since $E_2$ is a 0-dimensional sheaf at degree 0, $E_2^\vee[2] \in \Coh^{\leq 0} (X) [-1]$.  On the other hand, since $E_1$ is a sheaf at degree $-1$, the complex $E_1^\vee [1]$ sits at degrees 0 through 3.  The long exact sequence of cohomology of \eqref{eq29} then looks like
\[
 0 \to H^0(E^\vee [2]) \to H^{-1}(E_1)^\ast \to \HHom (E_2, \OO_X) \to \cdots.
\]
By our assumption on $ch(E)$, we have
\[
ch_i (H^{-1}(E_1)^\ast)= ch_i (H^0 (E^\vee [2])) = ch_i (I_Z \otimes L) \text{ for $i=0,1,2$}.
\]
Hence $ch (H^{-1}(E_1)^\ast \otimes L^\ast)$ is of the form $(1,0,0,\ast)$.  Since $H^{-1}(E_1)^\ast \otimes L^\ast$ is a reflexive sheaf, by our assumption on $X$ and \cite[Theorem 2]{Sim92}, this forces $H^{-1}(E_1)^\ast \otimes L^\ast \cong \OO_X$.  Hence $H^0(E^\vee [2]) = I_Y \otimes L$ for some 0-dimensional subscheme $Y \subset X$, while $H^0(E) = H^0(E_2) \in \Coh^{\leq 0}(X)$ as wanted.
\end{proof}

%[question: is the moduli of $Z_{\omega,B}$-semistable objects in $\Aob$ of Chern character $(1,0,0,\ast)$ exactly the Hilbert scheme of points? or, at least, are ideal sheaves of 0-dimensional subschemes of $X$ (maybe we need to twist by a line bundle) Bridgeland semistable objects?]

%Proposition \ref{prop.discr-0-nu-stable} also gives us a connection between tilt-semistable objects and semistable objects with respect to polynomial stabilities [give references]:
%
%\begin{lemma}\label{lemma4}
%Let $B=0$.  An object $E \in D(X)$ with nonzero rank, $\omega^2 ch_1 (E) > 0$ and $\overline{\Delta}_\omega (E)=0$ is $\nu$-semistable if:
%\begin{enumerate}
%\item $E$ is PT-semistable, or
%\item $E$ is $\sigma_3$-semistable.
%\end{enumerate}
%\end{lemma}
%
%[define PT-stability and $\sigma_3$-stability somewhere, or at least briefly explain what they are or give references]
%
%\begin{proof}
%Take any $E \in D(X)$ with Chern character satisfying the assumptions.  If $E$ is PT-semistable, then $H^{-1}(E)$ is $\mu$-semistable, $H^0(E)$ lies in $\Coh^{\leq 0}(X)$ and $\Hom (\Coh^{\leq 0}(X),E)=0$ by \cite{Lo1}.  Hence $E$ is $\nu$-semistable by Proposition \ref{prop.discr-0-nu-stable}.
%
%If $E$ is $\sigma_3$-semistable, then $H^{-1}(E)$ is $\mu$-semistable, $H^0(E) \in \Coh^{\leq 1}(X)$ and $\Hom (\Coh^{\leq 1}(X), E)=0$ by \cite{Lo3}, and so $E$ is $\nu$-semistable by Proposition \ref{prop.discr-0-nu-stable}.
%\end{proof}

\section{Tilt-semistable objects for $\omega \to \infty$}\label{sec.infinity}

In \cite[Section 7.2]{BMT}, Bayer-Macr\`{i}-Toda consider a subcategory $\mathfrak{D} \subset \Bob$ when $\omega$ is an ample $\mathbb{Q}$-divisor, where $\mathfrak{D}$ consists of objects $E \in\Bob$ of the following form:
\begin{itemize}
\item[(a)] $H^{-1}(E)=0$, and $H^0(E)$ is a pure sheaf of dimension $\geq 2$ which is slope semistable with respect to $\omega$.
\item[(b)] $H^{-1}(E)=0$, and $H^0(E) \in \Coh^{\leq 1}(X)$.
\item[(c)] $H^{-1}(E)$ is a torsion-free slope semistable sheaf, and $H^0(E) \in \Coh^{\leq 1}(X)$;  if $\muob (H^{-1}(E))<0$, then also $\Hom (\Coh^{\leq 1}(X),E)=0$.
\end{itemize}
And we have:
\begin{lemma}\cite[Lemma 7.2.1]{BMT}\label{lemma15}
If $E \in \Bob$ is $\nu_{m\omega,B}$-semistable for $m \gg 0$, then $E \in \mathfrak{D}$.
\end{lemma}

\begin{remark}\label{remark3}
We point out that any dual-PT-semistable complex (e.g.\ those termed as $\sigma_3$-semistable in \cite{Lo3}) of positive degree is of type (c) in the category $\mathfrak{D}$ above.  We do not know whether all dual-PT-semistable complexes of positive degree are $\nu_{m\omega, B}$-semistable for $m \gg 0$, although we take one step in this direction in Lemma \ref{lemma8} below.
\end{remark}

In this section, we try to prove the converse of Lemma \ref{lemma15}, which would give examples of tilt-stable objects when $\omega \to \infty$.  Since tilt-semistable objects with $\nuob=0$ are $Z_{\omega,B}$-semistable objects of phase 1 in $\Aob$, these results can help us describe Bridgeland semistable objects on threefolds as $\omega \to \infty$.

To start with, we observe the following easy consequence of Lemma \ref{lemma15} and Theorem \ref{theorem2}:

\begin{lemma}
Suppose $E\in \Bob$ is such that $\overline{\Delta}_\omega (E)=0$, $ch_0(E)<0$, $c_1(E)$ is proportional to $\omega$ and $\omega^2 \tch_1 (H^{-1}(E))<0$.  If $E$ is $\nu_{m\omega,B}$-semistable for $m \gg 0$, then $E = H^{-1}(E)[1]$ where $H^{-1}(E)$ is a $\muob$-semistable sheaf.
\end{lemma}

%[Also point out the connections between objects in this section with objects that are mapped to Gieseker stable sheaves under the Fourier-Mukai transforms constructed by Bridgeland-Maciocia.]

The next lemma is one step towards the converse of Lemma \ref{lemma15} for objects of type (c) above:

\begin{lemma}\label{lemma8}
Suppose $E \in \Bob$ satisfies the following: $H^{-1}(E)$ is a torsion-free slope stable sheaf, $H^0(E) \in \Coh^{\leq 1}(X)$, $\muob (H^{-1}(E)) < 0$ and $\Hom (\Coh^{\leq 1}(X),E)=0$. Then for any short exact sequence in $\Bob$
\begin{equation}\label{eq15}
0 \to M \to E \to N \to 0
\end{equation}
where $M,N \neq 0$, we have $\numob (M) < \numob (N)$ for $m \gg 0$.
\end{lemma}

Note that, Lemma \ref{lemma8} does not necessarily imply $E$ is $\numob$-stable for $m \gg 0$, since $m$ might depend on the particular short exact sequence \eqref{eq15} being considered.  To show that such $E$ is $\numob$-stable for $m \gg 0$, one might need to bound the Chern classes of all the $M$ or $N$ that appear in such short exact sequences, as is done in \cite[Theorem 1.1(ii)]{LQ}.

%[\textbf{check whether} the above lemma holds when we replace $\Coh^{\leq 1}$ by $\Coh^{\leq 0}$]

Before we prove Lemma \ref{lemma8}, let us make some observations:
\begin{itemize}
\item[(i)] The category $\Bob$ is invariant under replacing $\omega$ by $m\omega$ for any $m>0$.
\item[(ii)] If $A, C$ are two objects in $\Bob$ such that $\omega \tch_1(A), \tch_1(C) \neq 0$, then we have
    \begin{equation}\label{eq16}
      - \frac{1}{\muob (A)} < -\frac{1}{\muob (C)} \text{ if and only if } \nu_{m\omega, B} (A) < \nu_{m\omega, B} (C) \text{ for $m \gg 0$}.
    \end{equation}
    This is immediate from the equation
    \begin{equation}\label{eq19}
      \nu_{m\omega, B}(-) = \frac{m \omega \tch_2(-) - \frac{m^3 \omega^3}{6} \tch_0(-)}{m^2 \omega^2 \tch_1(-)}.
    \end{equation}
\end{itemize}

\begin{proof}[Proof of Lemma \ref{lemma8}]
Consider a short exact sequence \eqref{eq15} where $M, N \neq 0$. To  show that $\numob (M) < \numob (N)$ for $m \gg 0$, let  us divide into two cases:

\paragraph{Case 1:  $H^{-1}(M) \neq 0$.}  By the $\muob$-stability of $H^{-1}(E)$ and the assumption that $\muob (H^{-1}(E))<0$, we have $\omega^2 \tch_1(H^{-1}(M)) < 0$.  This implies $\omega^2 \tch_1 (M) >0$, and so   $\numob (M) < +\infty$ for all $m>0$.   If $\omega^2 \tch_1(N)=0$, then $\numob (N) = +\infty$ for all $m>0$, and so we have $\numob (M) < \numob (N)$ for all $m>0$.  For the remainder of Case 1, let us assume that $\omega^2\tch_1(N) \neq 0$.  Consider the long exact sequence of \eqref{eq15}:
\begin{equation}
0 \to H^{-1}(M) \overset{\al}{\to} H^{-1}(E) \overset{\beta}{\to} H^{-1}(N) \overset{\gamma}{\to} H^0(M) \overset{\delta}{\to} H^0(E) \to H^0(N) \to 0.
\end{equation}

Suppose $\image \gamma =0$.  Then we have $\muob (H^{-1}(M)) < \muob (H^{-1}(N)) <0$, implying $\numob (M) < \numob (N)$ for $m \gg 0$ by \eqref{eq16}.  If $\image \gamma \neq 0$, then we have $\muob (H^{-1}(M)) < \muob (\image \beta)$ as well as
 \begin{gather}
  \muob (\image \beta ) \leq \muob (H^{-1}(N)) \leq 0 \leq \muob (\image \gamma)
 \end{gather}
 by the see-saw principle.  Hence $\muob (H^{-1}(M)) < \muob (H^{-1}(N)) < 0$, and we have
 \begin{equation}\label{eq17}
   \numob (H^{-1}(M)) < \numob (H^{-1}(N)) \text{ for $m \gg 0$}
 \end{equation}
by \eqref{eq16}.  Note that both sides of \eqref{eq17} are $O(m)$ in magnitude.

Now, we have
\begin{equation}\label{eq18}
\numob (H^{-1}(N)) = \numob (N) - \frac{m \omega \tch_2 (H^0(N))}{m^2 \omega^2 \tch_1 (H^{-1}(N)[1])}.
\end{equation}
  On the other hand,
%\begin{align*}
%  \numob (M) &= \frac{m\omega \tch_2(M) - \frac{m^3\omega^3}{6} \tch_0(M)}{m^2 \omega^2 \tch_1 (M)} \\
%   &= \frac{m\omega \tch_2(M) - \frac{m^3\omega^3}{6} \left(\tch_0(H^{-1}(M)[1])+\tch_0 (H^0(M))\right)}{m^2 \omega^2 \left( \tch_1 (H^{-1}(M)[1])+\tch_1(H^0(M))\right)},
%\end{align*}
%giving us
\begin{multline}
\numob (M)
\leq \numob (M)  + \frac{\frac{m^3\omega^3}{6} \tch_0 (H^0(M))}{m^2 \omega^2 \left( \tch_1 (H^{-1}(M)[1])+\tch_1(H^0(M))\right)} \notag\\
= \frac{m\omega \tch_2(M) - \frac{m^3\omega^3}{6} \tch_0(H^{-1}(M)[1])}{m^2 \omega^2 \left( \tch_1 (H^{-1}(M)[1])+\tch_1(H^0(M))\right)} \notag\\
\leq \frac{m\omega \tch_2(M)}{m^2 \omega^2 \left( \tch_1 (H^{-1}(M)[1])+\tch_1(H^0(M))\right)} - \frac{\frac{m^3\omega^3}{6}\tch_0(H^{-1}(M)[1])}{m^2 \omega^2 \tch_1 (H^{-1}(M)[1])} \notag\\
= \frac{m\omega \tch_2(M)}{m^2 \omega^2 \left( \tch_1 (H^{-1}(M)[1])+\tch_1(H^0(M))\right)} \notag\\
- \frac{m\omega \tch_2(H^{-1}(M)[1])}{m^2 \omega^2 \tch_1 (H^{-1}(M)[1])} + \numob (H^{-1}(M)[1]).
\end{multline}
Letting $m \to \infty$ in the above inequalities while noting $\numob (H^{-1}(M)[1]) = \numob (H^{-1}(M))$, together with \eqref{eq17} and \eqref{eq18}, we obtain $\numob (M) < \numob (N)$ for $m \gg 0$.  This completes the proof of Case 1.

 \paragraph{Case 2: $H^{-1}(M) =0$.}  In this case, if $\image \gamma =0$, then $M=H^0(M) \in \Coh^{\leq 1}(X)$, contradicting our assumption $\Hom (\Coh^{\leq 1}(X),E)=0$.  So suppose $\image \gamma \neq 0$.

 If $\rank (H^0(M)) \neq 0$, then $\tch_1(H^0(M))>0$ by the definition of $\Tob$, and we have $\numob (M) = \numob (H^0(M)) <0$ for $m \gg 0$ from \eqref{eq19}, while $\numob (N) >0$ for $m \gg 0$.  That is, $\numob (M) < \numob (N)$ for $m \gg 0$.  Now, suppose  $\rank (H^0(M)) =0$ instead.

 If $\image \gamma \in \Coh^{\leq 1}(X)$, then since $H^{-1}(N)$ is torson-free, we obtain a nonzero class in $\Ext^1 (\image \gamma, H^{-1}(E)) \cong \Hom (\image \gamma, H^{-1}(E)[1])$, again  contradicting our assumption $\Hom (\Coh^{\leq 1}(X), E)=0$.

 If $\image \gamma$ is supported in dimension 2, then so is $H^0(M)$, and so  $\numob (M) = \numob (H^0(M)) \to 0$ as $m \to \infty$, while $\numob (N)>0$ for $m \gg 0$ from \eqref{eq19}.  Hence $\numob (M) < \numob (N)$ for $m \gg 0$.  This completes Case 2.
\end{proof}

The following lemma and corollary are more concrete than Lemma \ref{lemma8} - it tells us that line bundles are $\nuob$-stable when $\omega \to \infty$:

\begin{lemma}\label{lemma9}
Let $E$ be a line bundle with $\omega^2 \tch_1(E)<0$.  Then there exists a constant $m_0>0$, depending only on $c_1(E)$, such that $E[1]$ is $\numob$-stable whenever $m>m_0$.
\end{lemma}

\begin{proof}
To prove the lemma, it suffices to find a constant $m_0>0$, depending only on $ch(E)$, such that for every short exact sequence in $\Bob$
\begin{equation}\label{eq20}
0 \to M \to E[1] \to N \to 0
\end{equation}
where $M$ is a maximal destabilising subobject of $E[1]$ with respect to $\numob$ for some $m>0$, we have $\numob (M) < \numob (E[1])$ for $m>m_0$.

The long exact sequence of cohomology of \eqref{eq20} is
\begin{equation}
 0 \to H^{-1}(M) \overset{\al}{\to} E \overset{\beta}{\to} H^{-1}(N) \overset{\gamma}{\to} H^0(M) \to 0.
\end{equation}
If $H^{-1}(M)$ is of rank 1, then $\beta$ is the zero map, meaning $H^{-1}(N) \cong H^0(M)$.  This forces $N=0$, contradicting our assumption.  Hence $H^{-1}(M)$ must be zero.

If $\omega^2 \tch_1(M)=0$, then $M=H^0(M)$ must lie in $\Coh^{\leq 1}(X)$, giving us a subobject of $E$ that lies in $\Coh^{\leq 1}(X)$; this contradicts $\Ext^1 (\Coh^{\leq 1}(X), E)=0$.  Hence $\omega^2 \tch_1(M)>0$.  Then, since we are assuming $M$ is destabilising, we have $\numob (N) <\infty$, and so $\omega^2 \tch_1(N)>0$.

Since $M=H^0(M)$ is $\numob$-semistable for some $m$, by \cite[Corollary 7.3.2]{BMT} we have $\overline{\Delta}_{m\omega} (H^0(M)) \geq 0$, i.e.\
\[
(\omega^2 \tch_1 (H^0(M)))^2 \geq 2 \omega^3 \tch_0(H^0(M)) \omega \tch_2 (H^0(M)),
\]
which gives
\begin{equation}\label{eq21}
  \frac{\omega \tch_2 (H^0(M))}{\omega^2 \tch_1 (H^0(M))} \leq \frac{\omega^2 \tch_1 (H^0(M))}{2 \omega^3 \tch_0 (H^0(M))} = \frac{1}{2 \omega^3} \mu_\omega (H^0(M)).
\end{equation}
On the other hand, if we let $\delta = \tch_1(E)$, then since $\tch_1 (H^{-1}(N)) =  \delta + \tch_1 (H^0(M))$, we have
\begin{equation}\label{eq26}
0 < \omega^2 \tch_1 (H^0(M)) = \omega^2 \tch_1 (H^{-1}(N)) - \omega^2 \delta  < - \omega^2 \delta.
\end{equation}
Combining this with \eqref{eq21}, we get
\begin{equation}\label{eq22}
\frac{\omega \tch_2 (H^0(M))}{\omega^2 \tch_1 (H^0(M))}  < -\frac{\omega^2 \delta}{2 \omega^3}.
\end{equation}
Hence, when $m \geq 1$, we have
\begin{align}
  \numob (M)  &= \numob (H^0(M)) \notag \\
  &= \frac{m\omega \tch_2 (H^0(M)) - \frac{m^3 \omega^3}{6} \tch_0 (H^0(M))}{m^2 \omega^2 \tch_1 (H^0(M))}   \notag \\
  &< -\frac{\omega^2 \delta}{m2 \omega^3} - m\frac{\omega^3 \tch_0 (H^0(M))}{6\omega^2 \tch_1 (H^0(M))}   \notag \\
   &\leq -\frac{\omega^2 \delta}{2 \omega^3} - m\frac{\omega^3 \tch_0 (H^0(M))}{6\omega^2 \tch_1 (H^0(M))}   \notag \\
  &\leq -\frac{\omega^2 \delta}{2 \omega^3}. \label{eq24}
\end{align}
Since
\[
\numob (E) = \frac{m\omega \tch_2 (E) - \frac{m^3 \omega^3}{6} \tch_0 (E)}{m^2 \omega^2 \tch_1 (E)},
\]
it is clear that, there is a constant $m_0>0$ depending only on $ch(E)$, hence only on $c_1(E)$, such that $\numob (M) < \numob (E[1])$ whenever $m>m_0$.  This implies that $\numob (M) < \numob (N)$ whenever $m>m_0$, i.e.\ $E[1]$ is $\numob$-stable whenever $m>m_0$.
\end{proof}

%\begin{coro}\label{coro2}
%Let $E$ be a line bundle on $X$ with $\omega^2 \tch_1 (E)>0$.  Then there exists a constant $m_0 >0$, depending only on $c_1(E)$, such that $E$ is $\numob$-semistable whenever $m>m_0$.
%\end{coro}
%[\textbf{check:} I think we should be able to say `$\nu$-stable above: we just need to check that \cite[Proposition 5.1.3(ii)]{BMT} holds for stable objects, too.]
%
%\begin{proof}
%Since $\omega^2 \tch_1 (E)\neq 0$, we have $\numob (E) < \infty$, and so by the proof of \cite[Proposition 5.1.3]{BMT}, we know $E \in \Bob \cap \mathcal B^{\omega, B}$ and $\mathbb{D} (E) \in \mathcal B^{\omega, -B}$.  Since $E$ is a line bundle, $\mathbb{D} (E) = E^\ast [1]$ lies in the category $\mathcal F_1$ (defined on \cite[p.21]{BMT}).  Therefore, by \cite[Proposition 5.1.3(ii)]{BMT}, $E$ is $\numob$-semistable if and only if $\mathbb{D}(E)=E^\ast [1]$ is $\nu_{\omega, -B}$-semistable, which is true by Lemma \ref{lemma9}.
%\end{proof}

The following proposition computes an explicit bound for $m_0$ that appeared in Lemma \ref{lemma9}.  Part (c) of the proposition can also be used to verify the inequality in Conjecture \ref{conj.bmt.1.3.1}:

\begin{pro}\label{prop.m-stable-line-bundle}
Let $(X,\omega)$ be a polarised smooth projective threefold, and $m>0$.  Suppose $B=0$, $E$ is a line bundle on $X$, and let $d := c_1(E) \omega^2<0$.    Then for $m>0$,
\begin{itemize}
\item[(a)] $\numob (E[1])=0$ if and only if $m^2 = \frac{3 c_1(E)^2 \omega}{\omega^3}$.
\item[(b)] If $\numob (E[1])=0$, then $\numob (E[1])$ is $\numob$-stable whenever $m^2 \geq \frac{3d^2}{(\omega^3)^2}$.
\item[(c)] $\tch_3 (E[1]) < \frac{m^2\omega^2}{2} \tch_1 (E[1])$ is equivalent to $m^2 > \frac{c_1(E)^3}{3d}$.
\end{itemize}
%If $\omega =tH$ and $c_1(E)=sH$ for $s,t>0$ and some fixed ample divisor $H$ on $X$ with $H^3=1$, then the (in)equalities above translate to:
%\begin{equation*}
%(a): m^2 = 3s^2; \, (b): m^2 \geq \frac{3s^2}{t^2}; \, (c):  m^2 > \frac{s^2}{3t^2}.
%\end{equation*}
\end{pro}

Note that, if $c_1(E)$ is proportional to $\omega$, then $\overline{\Delta}_\omega (E)=0$, in which case equality holds in Conjecture \ref{conj.bmt.1.3.1} by results in \cite[Section 7.4]{BMT}.

\begin{proof}
(a) That $\numob (E[1])=0$ is equivalent to
\[
 m\omega ch_2 (E[1]) = \frac{m^3\omega^3}{6} ch_0 (E[1]),
\]
i.e.\ $m^2 \omega^3 = 3c_1^2 \omega$, and so the claim follows.

(b) Suppose $\numob (E[1])=0$.  From the proof of Lemma \ref{lemma9}, it suffices to show
\begin{equation}\label{eq25}
  - \frac{d}{m2\omega^3} - m \frac{ \omega^3 ch_0 (H^0(M))}{6 \omega^2 ch_1 (H^0(M))} \leq 0.
\end{equation}
whenever $m^2 \geq \frac{3d^2}{(\omega^3)^2}$, where $M$ is as in the inequalities \eqref{eq24}.

Now, from \eqref{eq26} we have
\[
  \frac{1}{\omega^2 ch_1 (H^0(M))} > - \frac{1}{d},
\]
and hence
\begin{align*}
  -\frac{d}{m2\omega^3} - m \frac{\omega^3 ch_0 (H^0(M))}{6\omega^2 ch_1 (H^0(M))} &<  -\frac{d}{m2\omega^3} + m \frac{\omega^3 ch_0 (H^0(M))}{6d} \\
  &\leq - \frac{d}{m2\omega^3} + \frac{m\omega^3}{6d}.
\end{align*}
Therefore, \eqref{eq25} holds if $- \frac{d}{m2\omega^3} + \frac{m\omega^3}{6d} \leq 0$, which is equivalent to $m^2 \geq \frac{3d^2}{(\omega^3)^2}$, and the claim follows.

(c) That $\tch_3 (E[1]) < \frac{m^2\omega^2}{2} \tch_1 (E[1])$ is equivalent to $- \frac{c_1(E)^3}{6} < - \frac{m^2 \omega^2 c_1(E)}{2}$, i.e.\ $c_1(E)^3 > 3dm^2$.  Since $d<0$, this is equivalent to $m^2 > \frac{c_1(E)^3}{3d}$ as claimed.
\end{proof}

%In Proposition \ref{prop.m-stable-line-bundle}, if $c_1(E)$ is proportional to $\omega$, then it is easy to use the (in)equalities in parts (a), (b) and (c) to verify that Conjecture \ref{conj.bmt.1.3.1} holds.  However, this is already pointed out in \cite[Section 7.4]{BMT}.

%[Check with Yogesh: is it possible to modify Example 7.6 in `results-1' so that we have an example for the above proposition, where $c_1(E)$ is not proportional to $\omega$?]

\section{Objects with twice minimal $\omega^2ch_1$}\label{sec.2c}

In \cite[Lemma 7.2.2]{BMT}, tilt-semistable objects $F$ with $\om^2\tch_1(F) \leq c$ are characterised, where
\begin{equation}\label{eqn-c}
c :=\min \{ \om^2\tch_1(F)>0 \mid F \in \Bob\}.
\end{equation}
 In the next proposition, we give some sufficient conditions for a torsion-free sheaf  $E\in \Tob$ with $\om^2\tch_1(E) =2c$ to be tilt-stable.

\begin{proposition}\label{lemma.2c}
Suppose $E \in \Tob$ is a torsion-free sheaf with $\nuob(E)=0$ and $\om^2\tch_1(E)=2c$, where $c$ is defined in  \eqref{eqn-c}.
\begin{enumerate}
\item \label{lemma.2c.part1} If  $\mu_{\om,B,\max}(E) < \frac{\om^3}{\sqrt{3}}$, then $E$ is $\nuob$-stable.

\item \label{lemma.2c.part2} If $\om^3>3\om(\tch_1(M))^2$ for every torsion free slope semistable sheaf $M$ with $\om^2\tch_1(M)=c$, then $E$ is $\nuob$-stable.
\end{enumerate}
\end{proposition}
\begin{proof}
Suppose we have a destabilizing short exact sequence in $\Bob$
\begin{equation}\label{eq27}
0 \to M \to E \to N \to 0
\end{equation}
with $\nuob(M) \geq \nuob(E)=0 \geq \nuob(N)$, and we may assume $M$ is $\nuob$-stable by replacing it with its maximal destabilizing subobject in $\Bob$ with respect to $\nuob$-stability.
The long exact sequence associated to \eqref{eq27} is
\begin{equation}\label{eqn30}
 0 \to H^{-1}(N) \overset{\al}{\to} H^0(M) \overset{\beta}{\to} E  \overset{\gamma}{\to} H^0(N) \to 0
\end{equation}
and we identify $M=H^0(M)$.

Since $\om^2\tch_1(E) = 2c$, the possibilities for $(\om^2\tch_1(M), \om^2\tch_1(N))$ are $(2c,0)$, $(c, c)$, and $(0,2c)$. The cases $(2c,0)$ and $(0,2c)$ are easily eliminated as possibilities as follows:

\begin{itemize}
\item case $(0,2c$):  Since $M=H^0(M) \in \Tob$, the condition $\omega^2\tch_1(M)=0$ forces $M$ to be torsion. Since $E$ is torsion free, we have $\beta =0$ in \eqref{eqn30}, hence $H^{-1}(N)\cong H^0(M)$, which forces $M=H^0(M)=0$, contrary to assumption.
\item case $(2c,0)$: in this case $\nuob(N)=\infty$ and equation \eqref{eqn30} cannot be a destabilizing sequence.
\end{itemize}

We now consider the case $(c,c)$. Since $M$ is $\nuob$-stable with $\om^2\tch_1(M)=c$, by \cite[Lemma 7.2.2]{BMT} we know that $H^0(M)$ lies in the set $\mathfrak{D}$ described in \cite[Section 7.2]{BMT}. Since $H^{-1}(N)$ and $E$ are torsion free, we have $M$ is torsion free, and by the description of elements of $\mathfrak{D}$ we have that $M$ is  a torsion-free slope semistable sheaf.

 Since $M$ is $\nuob$-stable, by \cite[Cor. 7.3.2]{BMT} we also have
\begin{equation}\label{eqn-bg2}
\om\tch_2(M)\leq \frac{(\om^2\tch_1(M))^2}{2\om^3\tch_0(M)}.
\end{equation}

Since $\nuob(E)=0$, the inequality $\nuob(M)\geq 0$ implies
\begin{equation}\label{eqn31}
\om\tch_2(M) \geq \frac{\om^3}{6}\tch_0(M).
\end{equation}
Combining  equation \eqref{eqn31} with equation \eqref{eqn-bg2} we get

\begin{equation}
\frac{\om^3}{6}\tch_0(M) \leq \frac{(\om^2\tch_1(M))^2}{2\om^3\tch_0(M)}
\end{equation}
or $\frac{\om^3}{\sqrt{3}} \leq \muob(M)$.

The hypothesis $\mu_{\om,B,\max}(E) < \frac{\om^3}{\sqrt{3}}$ implies $\mu_{\om,B,\max}(E) < \muob(M)$. Since $M$ is slope semistable, this inequality implies $\Hom_{\Coh(X)}(M, E)=0$ and hence $\beta=0$ in \eqref{eqn30}. We then get a contradiction as in the $(0,2c)$ case. This completes the proof of part \eqref{lemma.2c.part1}.

To prove part (2), suppose $E$ has a destabilising subobject $M$ as in \eqref{eq27}.  The usual Bogomolov-Giesker inequality gives us
\begin{equation}\label{eqn-bg}
\om\tch_2(M)\leq \frac{\om(\tch_1(M))^2}{2\tch_0(M)}.
\end{equation}
Combining \eqref{eqn-bg} with \eqref{eqn31}, we get
\begin{equation}
\frac{\om^3}{6}\tch_0(M) \leq \om\tch_2(M)\leq \frac{\om(\tch_1(M))^2}{2\tch_0(M)},
\end{equation}
and hence $\om^3(\tch_0(M))^2 \leq 3\om(\tch_1(M))^2$. Since $M$ is a torsion-free sheaf, we have $\tch_0(M) \geq 1$, and hence $\om^3 \leq 3\om(\tch_1(M))^2$.  Part (2) thus follows.
\end{proof}

%For every global section of a line bundle $L$, there is a natural map $L \otimes I_C \to L^2 \otimes I_C$, but it won't be destabilizing when $\numz(E)=0$, because in that case I have computed $\numz(L \otimes I_C) <0$.

In \cite[Example 7.2.4]{BMT},  Conjecture \ref{conj.bmt.1.3.1} was studied for rank-one sheaves of the form $E = L \otimes I_C$, where $L$ is a line bundle, $I_C$ the ideal sheaf of a curve on $X$, and $\omega^2 c_1(E)=c$. In the next proposition, following the ideas in \cite[Remark 2.10]{Toda}, we study  rank-one sheaves of the form $E=L^2 \otimes I_C$ where  $\omega^2 c_1(E)=2c$.  In particular, we apply Proposition \ref{lemma.2c} to find a condition when  $E$ is $\nuob$-semistable.  In part (4) of the proposition, we are able to verify Conjecture \ref{conj.bmt.1.3.1} for these particular objects $E$ by reducing the conjecture to the classical Castelnuovo inequality.

\begin{proposition}
Let $B=0$. Suppose $\Pic(X)$ is generated by an ample line bundle $L$ on $X$. Let $h:=c_1(L)$, $D:=h^3$, and $\om:=mh$ for some positive  $m \in \QQ$.  Suppose $C \subset X$ be a curve in $X$ of degree $d :=h\cdot [C]=h\cdot ch_2(\OO_C)$. Let $I_C$ be the ideal sheaf of $C \subset X$, and let  $$E:=L^2 \otimes I_C.$$
\begin{enumerate}
\item If $\nuz(E)=0$ then $m^2=12-\frac{6d}{D}$ and $d < 2D$.  The converse also holds.
\item If $\nuz(E)=0$ and $d<\frac{3}{2}D$, then $E$ is $\nuz$-stable.
\item If $-ch_3(\OO_C) \leq \frac{4}{3}d$ and $\nuz(E)=0$ then $E$ satisfies the inequality in Conjecture \ref{conj.bmt.1.3.1}.
\item If $d \leq D$, and $\nuz(E)=0$, and $X \subset \PP^4$ is a hypersurface of degree $D$, then $E$ satisfies the inequality in Conjecture \ref{conj.bmt.1.3.1}.
\end{enumerate}

\end{proposition}

\begin{proof} We follow the argument in \cite[Example 7.2.4]{BMT}.  To start with, note that
\begin{align*}
    ch_1(E) &=2h, \\
     ch_2(E) &=ch_0(L^2)ch_2(I_C) +ch_1(L^2)ch_1(I_C)  + ch_2(L^2)ch_0(I_C)=-[C] + 2h^2, \text{and} \\
ch_3(E)&=ch_3(L^2)ch_0(I_C) + ch_2(L^2)ch_1(I_C) + ch_1(L^2)ch_2(I_C) + ch_0(L^2)ch_3(I_C) \\
&= \frac{4D}{3}-2d  -ch_3(O_C).
\end{align*}

For part (1), note that  $\nuz(E)=0$ is equivalent to  $mh\cdot ch_2(E)=\frac{m^3h^3}{6}$, i.e.\ $-d + 2D = \frac{m^2D}{6}$, i.e.\  $m^2=12-\frac{6d}{D}$. Since $m^2>0$, it follows that   $d<2D$.

To prove part (2), we use of Proposition \ref{lemma.2c}. In our situation,  $c=\om^2h=m^2h^3$. Take any torsion-free slope semistable sheaf $M$ with $\omega^2 ch_1(M)=c$.  Then $ch_1(M)=h$.   By part \eqref{lemma.2c.part2} of Proposition \ref{lemma.2c}, $E$ would be $\nuob$-stable if we can show $\om^3 > 3\om(ch_1(M))^2$ i.e. $m^3h^3 > 3mh(h^2)$, or $m^2 > 3$; since $m^2=12-\frac{6d}{D}$, this is  equivalent to $d<\frac{3}{2}D$.

For part (3), just note that the inequality in Conjecture \ref{conj.bmt.1.3.1} now reads
 \begin{equation}\label{eqn.2c-conj}
\frac{4D}{3}-2d  -ch_3(O_C) \leq \frac{m^2\cdot 2D}{18}=\frac{4D}{3}-\frac{2d}{3}
\end{equation}
or, equivalently,
\begin{equation}
-ch_3(\OO_C) \leq \frac{4}{3}d.
\end{equation}
(Note that this is a stronger requirement than \cite[Equation (32)]{BMT}.)

For part (4), if $X \subset \PP^4$ is a hypersurface of degree $D$, then by Hirzebruch-Riemann-Roch we have
$$1-g = \chi(\OO_C) = ch_3(\OO_C) + \frac{d}{2}(5-D).$$ Then  \eqref{eqn.2c-conj} becomes
\begin{equation}\label{eqn.g-bound}
g \leq \frac{dD}{2}-\frac{7}{6}d+1.
\end{equation}

When $d\leq D$, the bound on $g$ in equation \ref{eqn.g-bound} follows from the Castelnuovo inequality $g \leq \frac{1}{2}(d-1)(d-2)$. (However in general, we only know $d <2D$, and also we do not know if $E$ is $\nuz$-stable.)
\end{proof}

\section{Tilt-unstable objects}\label{sec.miro-roig}

In this section, we use known inequalities between Chern characters of reflexive sheaves on $\PP^3$ to describe many slope stable reflexive sheaves $E \in \Bob$ that are tilt-unstable. We base our examples on the following result of Mir\'{o}-Roig:

\begin{proposition}\cite[Prop. 2.18]{MR1}\label{prop.MR-ch-0}
Let $X=\PP^3$ and $B=0$. For all $c_2, c_3$ such that $c_2 \geq 3$, $c_3$ is even, and $-c_2^2+c_2 \leq c_3\leq 0$, there exists a rank $3$ stable reflexive sheaf on $\PP^3$ with first through third Chern classes $(0, c_2, c_3)$.
\end{proposition}

\begin{proposition}
Let $X=\PP^3$, $\om=c_1(\OO(1))$, and $B=0$. Let $n$ and $m$ be positive integers of the same parity and with $3n^2 - m^2 \geq 6$. Let $E$ be a slope stable rank $3$ reflexive sheaf on $\PP^3$ with $c_1(E)=0$, $c_2(E)=\frac{3n^2-m^2}{2}$ and $c_3(E)$ an even integer satisfying $$- (2n^3 +\frac{2nm^2}{3} )> c_3(E) \geq -\frac{(9n^4-6n^2m^2 + m^4 - 6n^2+2m^2)}{4}$$ (such an $E$ exists by Proposition  \ref{prop.MR-ch-0}). Let  $F=E(-n)[1]$.
Then $F$ is tilt-unstable.
\end{proposition}

\begin{proof}
The condition $\nu_{m\om,  0}(F)=0$, i.e.\ $ch_2(F)=\frac{m^2}{6}ch_0(F)$, is equivalent to $c_2(E)=\frac{3n^2-m^2}{2}$.

First, note that such an $E$ exists: to use Miro-Roig's result, we require $0 \geq c_3\geq -c_2^2+c_2$ and $c_2 \geq 3$. If $c_2=\frac{3n^2-m^2}{2}$, the first inequality becomes $0 \geq c_3 \geq -\frac{(9n^4-6n^2m^2 + m^4 - 6n^2+2m^2)}{4}$, and the second becomes $3n^2 - m^2 \geq 6$. So such an $E$ exists.  Let  $F=E(-n)[1]$.  Then $ch_3(F)=\frac{n^3}{2} - nc_2 - \frac{c_3}{2}$ where $c_i=c_i(E)$, and we have $\muob(F)=\frac{c_1-3n}{3}=-n$.  Since $n > 0$, we have $F \in \Bob$.

Now, we claim that  $ch_3(F) > \frac{m^2}{18}ch_1(F)$.  Observe that
\begin{align*}
ch_3(F) &> \frac{m^2}{18}ch_1(F)    \\
\Leftrightarrow\frac{n^3}{2} - nc_2 - \frac{c_3}{2} &>  \frac{3nm^2}{18}  \\
\Leftrightarrow\frac{n^3}{2} - n\frac{(3n^2-m^2)}{2} - \frac{c_3}{2} &>  \frac{3nm^2}{18} \\
\Leftrightarrow - (2n^3 +\frac{2nm^2}{3} ) &> c_3(E),
\end{align*}
which holds by assumption.  Since Conjecture \ref{conj.bmt.1.3.1} holds on $\mathbb{P}^3$, as is proved in \cite{Mac},  $F$ must be $\nu_{\om, 0}$-unstable.
\end{proof}


\begin{thebibliography}{99}

\bibitem[BBR]{FMNT} C. ~Bartocci, U. ~Bruzzo, D. ~ Hern\'{a}ndez-Ruip\'{e}rez, \emph{Fourier-Mukai and Nahm Transforms in Geometry and Mathematical Physics}, Progress in Mathematics, Vol. 276, Birkh\"{a}user, 2009.

\bibitem[BMT]{BMT} A. ~Bayer, E. ~Macr\`{i} and Y. ~Toda, \emph{Bridgeland stability conditions on threefolds I: Bogomolov-Gieseker type inequalities}, 2011.  Preprint.  arXiv:1103.5010v1 [math.AG]

\bibitem[Har]{SRS} R. ~Hartshorne, \emph{Stable reflexive sheaves}, Math. Ann., Vol. 254, pp. 121-176, 1980.

\bibitem[Huy]{Huy06} D. ~Huybrechts, \emph{Derived and Abelian equivalences of K3 surfaces}, J. Algebraic Geom., Vol. 17, pp. 375-400, 2008.

\bibitem[HL]{HL} D. ~Huybrechts and M. ~Lehn, \emph{The Geometry of Moduli Spaces of Sheaves}, second edition, Cambridge University Press, Cambridge, 2010.

\bibitem[Lan1]{Lan} A. Langer, \emph{Lectures on Moduli of Torsion Free Sheaves}, IMPANGA, 2007.

\bibitem[Lan2]{Lan09} A. Langer, \emph{On the S-fundamental group scheme}, Ann. Inst. Fourier (Grenoble), Vol. 61, pp. 2077-2119, 2011.

\bibitem[Laz]{Laz} R. ~Lazarsfeld, \emph{Positivity in Algebraic Geometry I}, Springer, 2004.

\bibitem[Lo1]{Lo1} J. ~Lo, \emph{Moduli of PT-semistable objects I}, J. Algebra, Vol. 339 (1), pp. 203-222, 2011.

\bibitem[Lo2]{Lo3} J. ~Lo, \emph{Polynomial Bridgeland stable objects and reflexive sheaves}, 2011. To appear in Math.\ Res.\ Lett.\ 	 arXiv:1112.4511v1 [math.AG]

\bibitem[Lo3]{Lo7} J. ~Lo, \emph{Stable complexes and Fourier-Mukai transforms on elliptic fibrations}, 2012. Preprint. arXiv:1206.4281v1 [math.AG]

\bibitem[LQ]{LQ} J. ~Lo and Z. ~Qin, \emph{ Mini-walls for Bridgeland stability conditions on the derived category of sheaves over surfaces}, 2011. Preprint.  arXiv:1103.4352v1 [math.AG]

\bibitem[Mac]{Mac}  E. ~Macr\`{i},  \emph{A Generalized Bogomolov-Gieseker  inequality for the three dimensional projective space}, 2012.  Preprint.  arXiv:1207.4980v1 [math.AG]

\bibitem[Mat]{Mat} H. Matsumura. \emph{Commutative Ring Theory}, Cambridge University Press, 1986.

\bibitem[Mir]{MR1} R. Mir\'{o}-Roig. \emph{Chern Classes of Rank 3 Reflexive Sheaves}. Math. Ann. 276, 291-302 (1987).


\bibitem[Sim]{Sim92} C. Simpson, \emph{Higgs bundles and local systems}, IHES,  75 (1992) 5-95.

\bibitem[Tod]{Toda}  Y. ~Toda, \emph{A note on Bogomolov-Gieseker type inequality for Calabi-Yau 3-folds}, 2012.  Preprint.  arXiv:1201.4911v1 [math.AG]


\bibitem[Ver]{Ver07} P. Vermeire. \emph{Moduli of reflexive sheaves on smooth projective 3-folds}, Journal of Pure and Applied Algebra 211 (2007) 622-632.
\end{thebibliography}
\end{document}